\documentclass[english,11pt]{amsart}
\usepackage{amsopn,amsthm,amsfonts,amsmath,amssymb,amscd}
\usepackage{babel}
\usepackage{nicematrix}
\usepackage{mathtools}
\usepackage{tikz}

\newtheorem{Theorem}{Theorem}[section]
\newtheorem{Lemma}[Theorem]{Lemma}

\newtheorem*{Remark}{Remark}

\newenvironment{Proof*}{{\it Proof.}}{\hfill $\blacksquare$}

\newcommand{\NN}{\mathbb{N}}

\newcommand{\RR}{\mathbb{R}}
\newcommand{\ZZ}{\mathbb{Z}}

\newcommand{\ad}{{\rm ad}}

\newcommand{\Tr}{\mathrm{tr}}

\newcommand{\OO}{\mathcal{O}}

\newcommand{\im}{{\rm im}}

\begin{document}

\title{Unipotent orbits of elements in a quaternion ring of odd order}

\author{David Dol\v{z}an}
\address{Department of Mathematics, Faculty of Mathematics and Physics, University of Ljubljana, Jadranska 19, SI-1000 Ljubljana, Slovenia; and Institute of Mathematics, Physics and Mechanics, Jadranska 19, SI-1000 Ljubljana, Slovenia}
\email{david.dolzan@fmf.uni-lj.si}

\thanks{The author acknowledges the financial support from the Slovenian Research Agency  (research core funding No. P1-0222)}

\begin{abstract} 
  Let $n \in \NN$ and let $q=p^r$ be an odd prime power. Let $R$ be a finite commutative local principal ring of cardinality $q^{n}$ with $R/J(R) \simeq GF(q)$. We study the conjugation action of the group of all unipotent elements in the quaternion ring $H(R)$ on $H(R)$ and we classify the resulting
unipotent similarity classes, using a reduction to the ring of $2$-by-$2$ matrices over $R$.
\end{abstract}

\maketitle 

 \section{Introduction}

\bigskip

Let $R$ be a commutative ring with identity. The set
\begin{equation*}
H(R)= \{r_1+ r_2i + r_3j + r_4k : r_i \in R\}=R \oplus Ri \oplus Rj \oplus Rk,
\end{equation*}
together with the relations $ i^2 = j^2 = k^2 = ijk = -1$, and $ij = -ji$ turns out to be a ring, called the $\it{quaternion}$ $\it{ring}$ over $R$. This is obviously a generalization of Hamilton's division ring of real quaternions, $H(\RR)$. Recently, many studies have been devoted to quaternion rings and their properties. Among others, the authors in \cite{aris1, aris2} studied the ring $H(\ZZ_p)$, for a prime number $p$, while in \cite{mig2,mig1} the structure of $H(\ZZ_n)$ was studied.
In \cite{ghara}, the authors investigated the structure of superderivations of quaternion rings over $\ZZ_2$-graded rings, while in \cite{ghara2} the authors studied the mappings on quaternion rings.
In \cite{cher22}, the structure of the ring $H(R)$ was described for a ring $R$, where $2$ is an invertible element, the authors in \cite{xie} studied systems of matrix equations over $H(R)$, and in \cite{cherdol}, the representations of elements of a quaternion ring as sums of exceptional units were investigated.

In \cite{caluga}, the author has studied a problem of decomposing elements as a product of two idempotents in the ring of $2$-by-$2$ matrices over a general domain, and later in \cite{caluga1} the author also examined the question of which $2$-by-$2$ idempotent matrices are products of two nilpotent matrices. Furthermore, the characterization of singular $2$-by-$2$ matrices over a commutative domain which are products of two idempotents or products of two nilpotents was investigated in \cite{caluga2}.
In \cite{dolzanid, dolzannil} the author investigated the product of idempotents and the product of nilpotents in quaternion rings.

The action of the group of unipotent matrices on the set of $2$-by-$2$ matrices over commutative domains was studied recently in \cite{calugauni1, calugauni2}.

In this paper, we investigate the conjugation action on $H(R)$ induced by the subgroup of unipotent elements in $H(R)$. 
More precisely, for a ring $S$ we let
\[
\mathcal{U}(S)=\{\,1+n \in S; n\in S\text{ is nilpotent}\},
\]
and we consider the action of $\mathcal{U}(H(R))$ on $H(R)$ given by conjugation,
\[
(b,a)\longmapsto bab^{-1}\qquad (b\in \mathcal{U}(H(R)),\ a\in H(R)).
\]
We call two elements $a,b\in H(R)$ \emph{unipotent similar} if they lie in the same orbit for this action.

The motivation for studying unipotent similarity is twofold. First, conjugation by unipotent
elements is a natural and tractable substitute for full unit conjugation in noncommutative
finite rings: it retains enough structure to produce a meaningful stratification of the ring
into similarity classes, while being sufficiently explicit to allow orbit-size calculations.
Second, recent work on unipotent similarity of matrices over commutative domains
(e.g. \cite{calugauni1,calugauni2}) suggests that the resulting orbit structure is governed by
a small collection of invariants and admits uniform counting formulas; the quaternion setting
provides a noncommutative test-bed where these phenomena persist in a controlled way.

Throughout the paper we assume that $R$ is a finite commutative local principal ring of odd
order $|R|=q^{n}$ with residue field $R/J(R)\cong GF(q)$. The hypothesis that $q$ is odd
ensures that $2\in R^*$, which is used repeatedly in the structural decomposition of
unipotent elements and in orbit computations. Moreover, in odd characteristic the quaternion
ring $H(R)$ admits a concrete $2$-by-$2$ matrix realization; consequently, orbit problems in
$H(R)$ can be transferred to orbit problems in $M_{2}(R)$ under unipotent conjugation.
This reduction allows us to combine matrix techniques with the ideal-theoretic filtration by
powers of the Jacobson radical.

Our main result gives a complete description of the unipotent orbit decomposition: we determine
all possible orbit sizes and the exact number of orbits of each size. The key inputs are:
(i) a factorization result showing that every unipotent element in $GL_{2}(R)$ is a product of
elementary unipotent matrices up to central factors; (ii) an explicit orbit analysis over the
residue field $GF(q)$ in terms of the discriminant of the characteristic polynomial; and
(iii) a lifting principle showing that, under a mild rank condition on $\im(\ad(\overline A))$,
orbit sizes grow uniformly when passing from $R/J(R)$ to $R/J(R)^{k}$.

The structure of the paper is as follows. In Section~2 we collect definitions and preliminary
results on finite local principal rings and unipotent elements. In Section~3 we determine the
unipotent orbits in $M_{2}(R)$ (and hence in $H(R)$) and prove the main counting theorem.

\bigskip

 \section{Definitions and preliminaries}
\bigskip

All rings in our paper will be finite rings with identity. 
For a ring $R$, $Z(R)$ will denote its centre and $N(R)$ will denote the set of all its nilpotents. The group of units in $R$ will be denoted by $R^*$ and the Jacobson radical of $R$ by $J(R)$. We shall often denote the mapping of an element $x \in R$ under the canonical projection $R \rightarrow R/J$ by $\overline{x}$.

We will denote the $2$-by-$2$ matrix ring with entries in a subring $S$ of the ring $R$ by $M_2(S)$, while the group of invertible matrices therein will be denoted by $GL_2(S)$.


For any $a, b \in R$, let $M(a,b)=\left(\begin{array}{cc}
a & b\\
0 & 0
\end{array}\right) \in M_2(R)$. The trace of a matrix $A \in M_2(R)$ will be denoted by $\Tr(A)$. 

For $A \in M_2(R)$, we shall use the mapping $\ad(A): M_2(R) \rightarrow M_2(R)$, defined by $(\ad(A))(Y)=YA-AY$. By the above convention, $\ad(\overline{A})$ will then denote the obvious mapping from $M_2(R/J(R))$ to $M_2(R/J(R))$. 

A subgroup $G$ of $GL_2(R)$ acts on $M_2(R)$ by conjugation. We will denote the orbit of an element $A \in M_2(R)$ for this action by $\OO_G(A)$.
Also, the centralizer of $A \in M_2(R)$ in $T \subseteq M_2(R)$ will be denoted by $C_T(A)$.

It follows from \cite[Theorem 2]{ragha} that every finite local ring has cardinality $p^{nr}$ for some prime number $p$ and some integers $n, r$. Furthermore, the Jacobson radical $J(R)$ is of cardinality $p^{(n-1)r}$ and the factor ring $R/J(R)$ is a field with $p^r$ elements (denoted $GF(p^r)$).

We also have the following lemma, which can be found in \cite{dolzanid}.

\begin{Lemma} \cite[Lemma 2.2]{dolzanid}
\label{pid}    
Let $R$ be a finite local principal ring of cardinality $q^{n}$, where $q=p^r$ for some prime number $p$ and integers $n, r$ with $R/J(R) \simeq GF(q)$. Then there exists $x \in J(R)$ such that $J(R)^k=(x^k)$ for every $k \in \{0,1,\ldots,n\}$. In particular, $|J(R)^k|=q^{n-k}$ for every $k \in \{0,1,\ldots,n\}$.
\end{Lemma}

The following lemma will be used frequently throughout the paper. The proof is straightforward.

\begin{Lemma}
\label{obvious}
   Let $R$ be a ring. Then $R^*+J(R) \subseteq R^*$. 
\end{Lemma}

An element $x$ of ring $R$ is called unipotent if $x - 1 \in N(R)$, or equivalently, 
$x = 1 + n$ for some $n \in N(R)$. Lemma \ref{obvious} shows that any unipotent element is also a  unit. We define $\mathcal{U}(R)=\{x \in R; x$ is unipotent$\}$. Furthermore, we have the following easy lemma.

\begin{Lemma}
\label{grupa}
    Let $R$ be a ring. Then $\mathcal{U}(R)$ is a subgroup of $R^*$.
\end{Lemma}

The following lemma is straightforward.

\begin{Lemma}
\label{mocsl}
   Let $F=GF(q)$. Then $|SL_2(F)|=q(q^2-1)$. 
\end{Lemma}

Matrices $A, B \in  M_n(R)$ are similar if there exists an invertible matrix $P \in M_n(R)$ such that $A=P^{-1}BP$, so one matrix is in the orbit of the other of the conjugate group action of the group $GL_n(R)$ on $M_n(R)$. In this paper, we will consider the conjugate group action of the group of all unipotent elements in $GL_n(R)$ on $M_n(R)$. We say that $A, B \in  M_n(R)$ are unipotent similar if they are unipotent conjugate in $M_n(R)$. 

Let $R$ be a ring. We shall call $U(x)=\left(\begin{array}{cc}
0 & x\\
0 & 0
\end{array}\right) \in M_2(R)$ and $L(y)=\left(\begin{array}{cc}
0 & 0\\
y & 0
\end{array}\right) \in M_2(R)$ elementary unipotents for any $x, y \in R$.

The next lemma is proven by a direct calculation.

\begin{Lemma}
    \label{elementarydiag}
    Let $R$ be a commutative ring and $u \in R^*$. Then $L(-u^{-1})U(u-1)L(1)U(u^{-1}-1)=\left(\begin{array}{cc}
u & 0\\
0 & u^{-1}
\end{array}\right)$.
\end{Lemma}

The next few technical lemmas will be useful.

\begin{Lemma}
  \label{nilpo}
    Let $R$ be a commutative local finite ring. Then $N=\left(\begin{array}{cc}
x & y\\
z & w
\end{array}\right) \in M_2(R)$ is a nilpotent if and only if $x+w, x^2+zw \in J(R)$.
\end{Lemma}
\begin{proof}
    Suppose that $x+w, x^2+zw \in J(R)$. Denote $x+w=j \in J(R)$, then $w^2+yz=(j-x)^2+yz  \in J(R)$ and therefore
    $N^2\in M_2(J(R))$. Since $J(R)$ is a nilpotent ideal, $N$ is nilpotent.
    On the other hand, $R/J$ is a finite field. Therefore, if $N$ is nilpotent, then $\overline{N}^2=0$, which implies that $x+w, x^2+zw \in J(R)$.
\end{proof}

\begin{Lemma}
\label{enaplusj}
    Let $R$ be a commutative local finite ring and $A \in M_2(R)$. Then $\det(1+A)=1+\Tr(A)+\det(A)$. Furthermore, if $N \in M_2(R)$ is nilpotent, then $\Tr(N),\det(N) \in J(R)$.
\end{Lemma}
\begin{proof}
    A direct calculation shows that $\det(1+A)=1+\Tr(A)+\det(A)$. Suppose $N=\left(\begin{array}{cc}
x & y\\
z & w
\end{array}\right) \in M_2(R)$ is a nilpotent matrix. By Lemma \ref{nilpo}, we have $\Tr(N)=x+w \in J(R)$. Obviously, $\det(N) \notin R^*$, so $\det(N) \in J(R)$ as well.
\end{proof}

\begin{Lemma}
\label{kvadrat}
    Let $R$ be a commutative finite ring such that $2 \in R^*$. Then for any $j \in J(R)$ there exists $x \in J(R)$ such that $1+j=(1+x)^2$.
\end{Lemma}
\begin{proof}
    Let $\phi:J(R) \rightarrow J(R)$ be a mapping defined with $\phi(x)=x^2+2x$.
    Define $u=2(1+2^{-1}x)$ and observe that for $x \in J$ we have $u \in R^*$.
    Since $\phi(x)=xu$, we see that $\phi$ is injective and therefore bijective.
    Therefore, for any $j \in J(R)$, there exists $x \in J(R)$ such that $\phi(x)=j$.
    But $\phi(x)=(1+x)^2-1$, so we have $1+j=(1+x)^2$.
\end{proof}

Now, we can prove our first theorem.

\begin{Theorem}
\label{elementarne}
    Let $R$ be a commutative local finite ring such that $2 \in R^*$ and suppose $A \in M_2(R)$ is unipotent. Then $A$ is a product of elementary unipotents and central elements in $M_2(R)$.
\end{Theorem}
\begin{proof}
    Denote $A=1+N$ for a nilpotent matrix $N=\left(\begin{array}{cc}
x & y\\
z & w
\end{array}\right)$. 

 Now, suppose firstly that $1+x \in R^*$. Then a direct calculation shows that $1+N=L(z(1+x)^{-1})\left(\begin{array}{cc}
1+x & 0\\
0 & \det(1+N)(1+x)^{-1}
\end{array}\right) U(y(1+x)^{-1})$. However, we have $\det(1+N) \in 1+J(R)$ by Lemma \ref{enaplusj} and therefore by Lemma \ref{kvadrat}, there exists $r \in 1+J(R) \subseteq R^*$ such that $\det(1+N)=r^2$. This implies that $\left(\begin{array}{cc}
1+x & 0\\
0 & \det(1+N)(1+x)^{-1}
\end{array}\right)=rI \left(\begin{array}{cc}
(1+x)r ^{-1} & 0\\
0 & r(1+x)^{-1}
\end{array}\right)$. Since $rI$ is central, Lemma \ref{elementarydiag} now concludes the proof in this case.

 On the other hand, if $1+x \notin R^*$ and $1+w \in R^*$, we denote $P=\left(\begin{array}{cc}
0 & 1\\
1 & 0
\end{array}\right)$ and observe that $P^{-1}AP=\left(\begin{array}{cc}
w & z\\
y & x
\end{array}\right)$. By the above, we see that $P^{-1}AP$ is a product of elementary unipotents and central elements. Since a conjugate of an elementary unipotent is again an elementary unipotent (and obviously conjugation preserves central elements), the statement holds in this case as well. 
 
 Finally, examine the case  $1+x,1+w \in J(R)$. By Lemma \ref{obvious}, $-(1-(1+w))=w \in R^*$. But $w-x=(1+w)-(1+x) \in J(R)$ and $w+x \in J(R)$ by Lemma \ref{nilpo}. Therefore $2w \in J(R)$, and since $2 \in R^*$, we have $w \in J(R)$, a contradiction. 
\end{proof}

\bigskip

 \section{Orbits}
\bigskip

In this section, we determine the sizes of unipotent orbits. Theorem \ref{elementarne} shows that we must firstly examine conjugations with elementary unipotents.

A direct calculation shows that the next lemma holds.

\begin{Lemma}
    \label{elementconj}
    Let $R$ be a commutative ring, $A=\left(\begin{array}{cc}
a & b\\
c & d
\end{array}\right) \in M_2(R)$ and $s,t \in R$. Then the following two statements hold.
    \begin{enumerate}
        \item 
        $U(t)AU(t)^{-1}=\left(\begin{array}{cc}
a+tc & b+t(d-a)-t^2c\\
c & d-tc
\end{array}\right)$.
        \item
        $L(s)AL(s)^{-1}=\left(\begin{array}{cc}
a-sb & b \\
c+s(a-d)-s^2b & d+sb
\end{array}\right)$.
    \end{enumerate}
\end{Lemma}

%

We also have the following structural lemma, the proof of which is implicit in the proof of \cite[Theorem 4.6]{dolzan1}.

\begin{Lemma}
    \label{structure}
    Let $R$ be a finite commutative local ring of order $q^n$, $R/J$ a field of order $q$ and $J^{n-1} \neq 0$. Then there exist $g \in R^*$ and $x \in J \setminus J^2$ such that $R=\{\sum_{i=0}^{n-1}\lambda_i x^i;  \lambda_i \in \{0, g, g^2, \ldots, g^{q-1} \} \text { for every } i \in \{0,1,\ldots,n-1\}\}$. Furthermore, if $g^k-g^l \in J$ for some $k,l \in \{0,1,\ldots,q-1\}$, then $k=l$.
\end{Lemma}

Now, we can start investigating the orbits under unipotent conjugation.

\begin{Lemma}
\label{uniform}
    Let $R$ be a finite commutative local principal ring of cardinality $q^{n}$, where $q=p^r$ for some odd prime number $p$ and integers $n, r$ with $R/J(R) \simeq GF(q)$.
    Choose $A=\left(\begin{array}{cc}
a & b\\
c & d
\end{array}\right) \in M_2(R)$ and suppose $d-a \in J^\alpha \setminus J^{\alpha +1}, b \in J^\beta \setminus J^{\beta +1}$ and $c \in J^\gamma \setminus J^{\gamma +1}$. Denote $\delta=\min\{\alpha, \beta, \gamma\}$. 
Then there exists $A'=d'I+\left(\begin{array}{cc}
a' & b'\\
c' & 0
\end{array}\right) \in \OO_{\mathcal{U}(M_2(R))}(A)$ such that $a' \in J^\delta \setminus J^{\delta +1}$.
\end{Lemma}
\begin{Remark}
    By Lemma \ref{pid}, we know that $J(R)^n=0$. By a slight abuse of notation, we shall abide by the convention that $J(R)^n\setminus J(R)^{n+1}$ equals $\{0\}$, since it simplifies many of our statements. We shall use this same convention throughout the remainder of this paper.
\end{Remark}
\begin{proof}
    If $A'=\left(\begin{array}{cc}
a' & b'\\
c' & d'
\end{array}\right)=U(t)AU(t)^{-1}$ then $a'-d'=a-d+2tc$, and if $A'=\left(\begin{array}{cc}
a' & b'\\
c' & d'
\end{array}\right)=L(s)AL(s)^{-1}$ then $d'-a'=d-a+2sb$, both by Lemma \ref{elementconj}. Since $p$ is odd, we have $2 \in R^*$. This implies that we can assume without any loss of generality that $\alpha \leq \beta, \gamma$, so $\delta=\alpha$. 
    
    

    We have therefore proved that there exists $A'=\left(\begin{array}{cc}
a' & b'\\
c' & d'
\end{array}\right) \in \OO_{\mathcal{U}(M_2(R))}(A)$ such that $a'-d' \in J^\delta \setminus J^{\delta +1}$. Therefore  $A'=\left(\begin{array}{cc}
d' & 0\\
0 & d'
\end{array}\right)+\left(\begin{array}{cc}
a'-d' & b'\\
c' & 0
\end{array}\right)$ and thus the lemma holds.
\end{proof}

The next lemma examines the connection between orbits of $A \in M_2(R)$ and the orbits of $\overline {A} \in M_2(R/J(R))$.

\begin{Lemma}
\label{prehodnagor}
Let $R$ be a finite local principal ring of cardinality $q^{n}$, where 
$q = p^{r}$ for some prime $p$ with $R/J(R) \cong GF(q)$.
For $k \in \{1,2,\ldots,n\}$, let
$\widetilde{A} \in M_2(R/J(R)^{k})$
denote the images of $A \in M_2(R)$ under the canonical projections.
Suppose that
$\dim_{R/J(R)}\!\bigl(\operatorname{im}(\operatorname{ad}(\overline{A}))\bigr) = 2$.
Then the conjugacy orbit of $\widetilde{A}$ under the group 
$\mathcal{U}(M_2(R/J(R)^{k}))$ satisfies
\[
\bigl|\mathcal{O}_{\mathcal{U}(M_2(R/J(R)^{k}))}(\widetilde{A})\bigr|
   = 
\bigl|\mathcal{O}_{\mathcal{U}(M_2(R/J(R)))}(\overline{A})\bigr|
\, q^{2(k-1)}.
\]
Moreover, the number of distinct $\mathcal{U}(M_2(\cdot))$–orbits in 
$M_2(R/J(R)^{k})$ is equal to
\[
q^{2(k-1)} \times 
\#\text{orbits in } M_2(R/J(R)).
\]
\end{Lemma}
\begin{proof}
  Denote $J(R)=(x)$. Observe firstly that every $A \in M_2(R)$ can be written by Lemma \ref{structure} as $A=A_0+xA_1+\ldots+x^{n-1}A_{n-1}$ for some $A_i \in M_2(F)$ (where $F$ is a fixed system of representatives
for the residue field $R/J(R)\cong \mathrm{GF}(q)$). Modulo $J(R)^{j+1}$, every element congruent to $A$ can be written in the form
\[
A + x^{j} Z, \qquad Z \in M_2(F), 
\]
since all higher powers of $x$ vanish in $J(R)^{j+1}$.
For $Y,Z \in M_2(F)$, we have 
$(I + x^{j}Y)^{-1} \equiv I - x^{j}Y \pmod{J(R)^{j+1}}$, and hence
\[
(I+x^{j}Y)(A+x^{j}Z)(I+x^{j}Y)^{-1}
\equiv A + x^{j}Z + x^{j}(\operatorname{ad}(A))(Y)
\pmod{J(R)^{j+1}}.
\]
Since $A \equiv A_0 \pmod{J(R)}$, the reductions of $\operatorname{ad}(A)$
and $\operatorname{ad}(A_0)$ coincide modulo $J(R)$, so
\[
x^{j}(\operatorname{ad}(A))(Y)
\equiv x^{j}(\operatorname{ad}(A_0))(Y)
\pmod{J(R)^{j+1}}.
\]
Let $A' = A + x^{j}Z_1$ and $A'' = A + x^{j}Z_2$.  
From the above computation, $A'$ and $A''$ are conjugate modulo $J(R)^{j+1}$
by some matrix of the form $I + x^{j}Y$ if and only if
\[
Z_2 - Z_1 \in \operatorname{im}(\operatorname{ad}(A_0))
\pmod{J(R)}.
\]
Passing to the residue field $R/J(R)$, this condition is equivalent to
\[
\overline{Z_2} - \overline{Z_1}
   \in \operatorname{im}(\operatorname{ad}(\overline{A})).
\]
By assumption, 
$\dim_{R/J(R)}\!\bigl(\operatorname{im}(\operatorname{ad}(\overline{A}))\bigr)=2$.
Thus
\[
|\operatorname{im}(\operatorname{ad}(\overline{A}))| = q^2,
\qquad
\bigl| M_2(R/J(R)) / \operatorname{im}(\operatorname{ad}(\overline{A})) \bigr|
   = q^2.
\]
Hence each orbit modulo $J(R)^{j}$ splits into exactly $q^2$ inequivalent lifts
modulo $J(R)^{j+1}$, and the size of each orbit is multiplied by $q^2$.
Iterating this lifting step for $j = 1,2,\dots,k-1$ yields
\[
\bigl|\mathcal{O}_{\mathcal{U}(M_2(R/J(R)^k))}(\widetilde{A})\bigr|
   = \bigl|\mathcal{O}_{\mathcal{U}(M_2(R/J(R)))}(\overline{A})\bigr|\,
     q^{2(k-1)}.
\]
The statement about the number of distinct orbits follows analogously
by counting orbit representatives instead of orbit sizes.
\end{proof}

We shall also need the following technical lemmas.

\begin{Lemma}
\label{disco}
    Let $F=GF(q)$ and $A=\left(\begin{array}{cc}
a & b\\
c & 0
\end{array}\right) \in M_2(F)$. Then $\Delta=a^2+4bc$ is the discriminant of the characteristic polynomial of $A$.    
\end{Lemma}
\begin{proof}
  Let $p(\lambda)=\det(A-\lambda I)=\lambda^2+a\lambda-bc$, so its discriminant is equal to $a^2+4bc$.
\end{proof}

\begin{Lemma}
\label{deltasq}
    Let $F=GF(q)$ be a Galois field of odd order. Let $p(t)= ct^2+at-b\in F[t]$ with $c \neq 0$, and let $\Delta(p)=a^2+4bc$ be its discriminant. Then $\Delta(p)$ is a square in $F$ if and only if $p(t)$ is not an irreducible polynomial in $F[t]$.    
\end{Lemma}
\begin{proof}
Obviously $p(t)= ct^2+at-b$ is irreducible in $F[t]$ if and only if $q(t)=t^2+ac^{-1}t-bc^{-1}$ is irreducible in $F[t]$ and $\Delta(p)$ is a square in $F$ if and only if $\Delta(q)=\Delta(p) c^{-2}$ is a square in $F$.

Hence we can assume that $c=1$.
If $p(t)$ is reducible then $p(t)=(t-r_1)(t-r_2)$ for some $r_1,r_2 \in F$. But then $a=-r_1-r_2$ and $b=-r_1r_2$, so $\Delta(p)=a^2+4b=(r_1-r_2)^2$ is a square in $F$.

Conversely, suppose that $\Delta(p)=a^2+4b=d^2$ for some $d \in F$. Then $(t-2^{-1}(d-a))(t+2^{-1}(d+a))=t^2+at-2^{-2}(d^2-a^2)=p(t)$ is reducible in $F[t]$.
\end{proof}

The following lemma examines the sizes of orbits over a field, and is a crucial lemma in our investigation.

\begin{Lemma}
\label{orbitmodulo}
    Let $F \simeq GF(q)$ be a field of odd order and let $A=\left(\begin{array}{cc}
a & b\\
c & 0
\end{array}\right) \in M_2(F)$, where $a \neq 0$. Denote $\phi(t)=ct^2+at-b\in F[t]$. Then we have the following statements.
\begin{enumerate}
    \item 
    \begin{equation*}
   |\OO_{\mathcal{U}(M_2(F))}(A)|=
    \begin{cases} 
    q(q+1), \text { if } \phi(t) \text { has a simple root in } F, \\
    \frac{q^2-1}{2}, \text { if } \phi(t) \text { has a double root in } F, \\
    q(q-1), \text { if } \phi(t) \text { does not have a root in } F.
    \end{cases}
    \end{equation*}
    
    \item 
    In $M_2(F)$ there are exactly $q$ unipotent orbits of order $1$, exactly $\frac{q(q-1)}{2}$ unipotent orbits of order $q(q+1)$, exactly $\frac{q(q-1)}{2}$ unipotent orbits of order $q(q-1)$ and exactly $2q$ unipotent orbits of order $\frac{q^2-1}{2}$.
    
    \item
    $\dim_F(\im(ad(A)))=2$.
\end{enumerate}
\end{Lemma}
\begin{proof}
  Since it is well known that $SL_2(F)$ is generated by elementary unipotents (see for example \cite[Lemma 8.1]{lang}), Theorem \ref{elementarne} tells us that $|\OO_{\mathcal{U}(M_2(F))}(A)|=|\OO_{SL_2(F)}(A)|$. Since $|SL_2(F)|=q(q^2-1)$ by Lemma \ref{mocsl}, and $|\OO_{SL_2(F)}(A)|=\frac{|SL_2(F)|}{|C_{SL_2(F)}(A)|}$, we have $|\OO_{\mathcal{U}(M_2(F))}(A)|=\frac{q(q^2-1)}{|C_{SL_2(F)}(A)|}$. Therefore, we have to find the size of the centralizer of $A$ in $SL_2(F)$.
    \begin{enumerate}
        \item 
        Let us firstly consider the case when $\phi(t)$ has a simple root in $F$, say $\omega_1$.
        If $c=0$, then conjugating with $U(\omega_1)$ shows that $A$ is diagonalizable over $F$ by Lemma \ref{elementconj}. If $c \neq 0$, then $\phi(t)$ has a nonzero root $\omega_2 \in F$. Therefore, $c\omega_2^2+a\omega_2-b=0$ and thus also $c+a\omega_2^{-1}-b\omega_2^{-2}=0$.
        This implies that $A$ is again diagonalizable over $F$ by conjugating with elementary unipotents by Lemma \ref{elementconj}. So, we can assume that $A=\left(\begin{array}{cc}
\alpha & 0\\
0 & \beta
\end{array}\right)$ for some $\alpha \neq\beta \in F$. 
        Therefore, $C_{SL_2(F)}(A)=\left\{\left(\begin{array}{cc}
\gamma & 0\\
0 & \gamma^{-1}
\end{array}\right); \gamma \in F \setminus \{0\}\right\}$, so $|C_{SL_2(F)}(A)|=q-1$ and consequently $|\OO_{SL_2(F)}(A)|=q(q+1)$.

Observe that $\phi(t)$ has a simple root in $F$ if and only if its discriminant $\Delta(\phi(t))=a^2+4bc$ is a nonzero square in $F$. By Lemma~\ref{disco}, $\Delta(\phi)$ equals the discriminant of the characteristic
polynomial of $A$, namely $\Tr(A)^2-4\det(A)$. However, $\Tr(A), \det(A)$ are invariants under conjugation. So, choose $\Tr(A)=\alpha \in F$ arbitrarily. Since $F$ is a field of odd order, the mapping $\rho: F \rightarrow F$, defined with $\rho(z)= \alpha^2 - 4 z$ is a bijection on $F$. Therefore there exist exactly $\frac{q-1}{2}$ elements $z \in F$ such that $\rho(z)$ is a nonzero square in $F$. This implies that we have at least
$\frac{q(q-1)}{2}$ different orbits in this case.

Since $A$ is diagonal, we have $C_{M_2(R)}(A)=F[A]$, so $\dim_F(\ker(ad(A)))=2$ and thus we also get $\dim_F(\im(ad(A)))=2$.

\item 
Suppose now that $\phi(t)$ has a double root in $F$, so $\phi(t)=c(t+\omega)^2$ for some $\omega \in F$. Obviously this means that $c \neq 0$. This also implies that $a=2c\omega$ and $b=-c\omega^2$, so $b=-a^2(4c)^{-1}$. It follows that $\det(A-\lambda I)=\lambda^2-a\lambda-bc=(\lambda - 2^{-1}a)^2$. Since $A$ is not a scalar matrix, we have that $A$ is similar to the matrix
$\left(\begin{array}{cc}
2^{-1}a & 1\\
0 & 2^{-1}a
\end{array}\right)$. However, Lemma \ref{elementconj} shows that we can achieve the upper triangular form of $A$ by conjugation with elementary unipotents, so we can assume that $A$ is unipotent similar to $\left(\begin{array}{cc}
2^{-1}a & r\\
0 & 2^{-1}a
\end{array}\right)$ for some $0 \neq r \in F$. Now, since a scalar matrix does not change neither the size of orbit of $A$ nor the dimension of $\im(ad(A))$, we can assume that $A=\left(\begin{array}{cc}
0 & r\\
0 & 0
\end{array}\right)$. Therefore, we have for $Y=\left(\begin{array}{cc}
\alpha & \beta \\
\gamma  & \delta
\end{array}\right) \in M_2(F)$ that $(ad(A))(Y)=r\left(\begin{array}{cc}
-\gamma & \alpha-\delta \\
0 & \gamma
\end{array}\right)$.
 So, $Y \in \ker(ad(A))$ implies $\gamma=0$ and $\alpha=\delta$ and consequently $\dim_F(\ker(ad(A)))=2$ and thus also $\dim_F(\im(ad(A)))=2$. Furthermore, we get that $C_{SL_2(F)}(A)=\left\{\left(\begin{array}{cc}
\alpha & \beta\\
0 & \alpha
\end{array}\right); \alpha, \beta \in F, \alpha^2=1\right\}=\pm I+\left\{\left(\begin{array}{cc}
0 & \beta\\
0 & 0
\end{array}\right); \beta \in F\right\}$, which implies that $|C_{SL_2(F)}(A)|=2q$, so the statement holds in this case as well.

If $A=\left(\begin{array}{cc}
0 & r\\
0 & 0
\end{array}\right)$ is unipotently similar to $B=\left(\begin{array}{cc}
0 & s\\
0 & 0
\end{array}\right)$, then there exists $C=\left(\begin{array}{cc}
\alpha & \beta\\
\gamma & \delta
\end{array}\right) \in SL_2(F)$ such that $CA=BC$, which implies that $\gamma=0, \delta=\alpha^{-1}$ and therefore $s=\alpha^2 r$. So, for every $\lambda \in F$, we have at least two orbits corresponding to $\left(\begin{array}{cc}
\lambda & r\\
0 & \lambda
\end{array}\right)$: one, if $r$ is a square in $F$, and one if $r$ is not a square. Thus, we have at least $2q$ orbits in this case.

\item 
Finally, let us consider the case when $\phi(t)$ does not have a root in $F$. Since $a \neq 0$, this means that $c \neq 0$. Lemma \ref{deltasq} yields that $\Delta=a^2+4bc$ is not a square in $F$. By Lemma \ref{disco}, $\Delta$ is the discriminant of the characteristic polynomial of $A$ and thus again by Lemma \ref{deltasq} this means that the characteristic polynomial of $A$ is irreducible over $F$. This means that $A$ is a nonderogatory matrix, which implies that 
$C_{M_2(F)}(A) \simeq F[A]$ (see for example \cite[\S 2.08]{wedderburn}). Since the minimal polynomial of $A$ is an irreducible polynomial of degree $2$, we have $F[A] \simeq GF(q^2)$. Consequently, $\dim_F(\ker(ad(A)))=2$ and therefore $\dim_F(\im(ad(A)))=2$. Furthermore, $C_{GL_2(F)}(A)=F[A] \setminus \{0\}$, so 
$|C_{GL_2(F)}(A)|=q^2-1$. Let $f: GF(q^2) \setminus \{0\} \rightarrow C_{GL_2(F)}(A)$, defined with $f(x)=m_x$, be the canonical group isomorphism (where $m_x$ is the matrix (in the chosen basis of $GF(q^2)$ over $GF(q)$, representing multiplication by $x$). If we restrict the determinant map from $GL_2(F)$ to $C_{GL_2(F)}(A)$, we have
$\det(f(x))=\det(m_x)$. By \cite[Corollary 11.81]{rotman}, we have $\det(m_x)=N_{GF(q^2)/GF(q)}(x)$ and since the Galois group of the field extension $GF(q^2)/GF(q)$ is equal to $\{id, u \mapsto u^q\}$, we have $\det(m_x)=x^{q+1}$.
Let us denote $\psi: GF(q^2) \setminus \{0\} \rightarrow GF(q^2) \setminus \{0\}$,
$\psi(x)=\det(f(x))=x^{q+1}$, which is a (multiplicative) group endomorphism and $C_{SL_2(F)}(A)$ corresponds exactly to $\ker(\psi)$. Let $g$ denote the multiplicative generator of $GF(q^2) \setminus \{0\}$. Then $\psi(g^k)=g^{k(q+1)}$, so $g^k \in \ker(\psi)$ if and only if $q^2-1$ divides $k(q+1)$ if and only if $q-1$ divides $k$.
Thus $\ker(\psi)=\{g^{q-1},g^{2(q-1)},\ldots,g^{(q+1)(q-1)}\}$ and thus  $|C_{SL_2(F)}(A)|=q+1$.

By Lemma \ref{deltasq}, $\phi(t)$ is irreducible over $F$ if and only if $\Delta(\phi(t))$ is not a square in $F$. By fixing $\Tr(A)=a \in F$ arbitrarily and observing the bijection $\rho: F \rightarrow F$, $\rho(z)= a^2 - 4 z$, we see similarly as in case (1) that there exist exactly $\frac{q-1}{2}$ elements $z \in F$ such that $\rho(z)$ is not a square in $F$, therefore we have at least
$\frac{q(q-1)}{2}$ different orbits in this case.
\end{enumerate}

Since scalar matrices in $M_2(F)$ form $q$ orbits of order $1$, and by above we have $\frac{q(q-1)}{2}$ orbits of order $q(q+1)$, $\frac{q(q-1)}{2}$ orbits of order $q(q-1)$ and $2q$ orbits of order $\frac{q^2-1}{2}$, and that $q \cdot 1 + \frac{q(q-1)}{2} \cdot q(q+1) + \frac{q(q-1)}{2} \cdot q(q-1) + 2q \cdot \frac{q^2-1}{2} = q^4$, we see that $M_2(F)$ is really covered by all the above orbits.
\end{proof}

Finally, we shall need the following lemma.

\begin{Lemma}
\label{enakeorbite}
    Let $R$ be a finite commutative local principal ring of cardinality $q^{n}$, where $q=p^r$ for some odd prime number $p$ and integers $n, r$ with $R/J(R) \simeq GF(q)$. Denote $J(R)=(x)$. Choose $a,a' \in R^*$, $b,c,b',c',d, d' \in R$, and $\delta, \delta' \in \{0,1,\ldots,n\}$. If $dI+x^\delta\left(\begin{array}{cc}
a & b\\
c & 0
\end{array}\right) \in \OO_{\mathcal{U}(M_2(R))}\left(d'I+x^{\delta'}\left(\begin{array}{cc}
a' & b'\\
c' & 0
\end{array}\right)\right)$, then $\delta = \delta'$ and $d-d' \in J(R)^\delta$.
\end{Lemma}
\begin{proof}
    Suppose that $dI+x^\delta\left(\begin{array}{cc}
a & b\\
c & 0
\end{array}\right) \in \OO_{\mathcal{U}(M_2(R))}\left(d'I+x^{\delta'}\left(\begin{array}{cc}
a' & b'\\
c' & 0
\end{array}\right)\right)$. This implies that $dI+x^\delta\left(\begin{array}{cc}
a & b\\
c & 0
\end{array}\right) = d'I+x^{\delta'}B$, for some $B \in \OO_{\mathcal{U}(M_2(R))}\left(\left(\begin{array}{cc}
a' & b'\\
c' & 0
\end{array}\right)\right)$. This implies that $d=d'+x^{\delta'}r$ for some $r \in R$. Since the trace of a matrix is preserved under conjugation, we have $2(d-d')=a'x^{\delta'}-ax^\delta$, therefore $2rx^{\delta'}=a'x^{\delta'}-ax^\delta$.
This implies that $\delta \geq \delta'$. Since we also have $d'I+x^{\delta'}\left(\begin{array}{cc}
a' & b'\\
c' & 0
\end{array}\right) \in \OO_{\mathcal{U}(M_2(R))}\left(dI+x^{\delta}\left(\begin{array}{cc}
a & b\\
c & 0
\end{array}\right)\right)$, this means that $\delta'=\delta$. This also shows that $d-d' \in J(R)^\delta$.
\end{proof}

We can now prove the main result of this section.

\begin{Theorem}
   Let $R$ be a finite commutative local principal ring of cardinality $q^{n}$, where $q=p^r$ for some odd prime number $p$ and integers $n, r$ with $R/J(R) \simeq GF(q)$. Then there are exactly $q^n$ unipotent orbits of order $1$,
   \begin{itemize}
       \item 
       $\frac{q^{2(n-\delta)-1}(q-1)}{2}$ unipotent orbits of order $q^{2(n-\delta)-1}(q+1)$,

       \item 
       $2q^{2(n-\delta)-1}$ unipotent orbits of order $\frac{q^{2(n-\delta-1)}(q^2-1)}{2}$, and

       \item 
       $\frac{q^{2(n-\delta)-1}(q-1)}{2}$ unipotent orbits of order $q^{2(n-\delta)-1}(q-1)$,
   \end{itemize}
   for every $\delta \in \{0,1,\ldots,n-1\}$, in $H(R)$.
 \end{Theorem}
\begin{proof}
    Since $p$ is odd, we know that $H(R) \simeq M_2(R)$ \cite[Theorem 3.10]{cher}, so we can study the $2$-by-$2$ matrix ring.
    Denote $J(R)=(x)$. 
    Choose $A \in M_2(R)$. Obviously, there are exactly $q^n$ scalar matrices, which each have a unipotent orbit of order $1$. So, assume from now on, that $A$ is not scalar.
    By Lemma \ref{uniform}, we can assume that $A=dI+x^\delta\left(\begin{array}{cc}
a & b\\
c & 0
\end{array}\right)$ for some $a \in R^*$, some $b,c, d \in R$ and some integer $\delta \in \{0,1,\ldots,n-1\}$. Then, for any invertible matrix $P \in M_2(R)$, we have $P^{-1}AP=dI+x^\delta P^{-1}\left(\begin{array}{cc}
a & b\\
c & 0
\end{array}\right)P$. Denote $B=\left(\begin{array}{cc}
a & b\\
c & 0
\end{array}\right)$, $\widetilde{B}=B+M_2(J(R))^{n-\delta}$ the mapping of $B$ under the canonical projection $M_2(R) \rightarrow M_2(R/J(R)^{n-\delta})$. By Lemma \ref{orbitmodulo}, we have  $\dim_F(\im(ad(\overline{B})))=2$. By Lemma \ref{prehodnagor}, we therefore have $|\OO_{\mathcal{U}(M_2(R/J(R)^{n-\delta}))}(\widetilde{B})|=|\OO_{\mathcal{U}(M_2(R/J(R)))}(\overline{B})| q^{2(n-\delta-1)}$. So,  $|\OO_{\mathcal{U}(M_2(R))}(A)| \in \{q^{2(n-\delta)-1}(q+1),\frac{q^{2(n-\delta-1)}(q^2-1)}{2},q^{2(n-\delta)-1}(q-1)\}$, again by Lemma \ref{prehodnagor}. 

Observe that as $\delta$ ranges from $0$ to $n-1$, all these numbers are distinct.
By Lemma \ref{enakeorbite}, we know that for every $\delta$, the choice of $d$ gives us at least $|R/J(R)^\delta|=q^{\delta}$ different orbits.
For a fixed $d$ and $\delta$, we know by Lemma \ref{prehodnagor} and Lemma \ref{orbitmodulo} that we have exactly $\frac{q^{2(n-\delta)-1}(q-1)}{2}$ unipotent orbits of order $q^{2(n-\delta)-1}(q+1)$, exactly $\frac{q^{2(n-\delta)-1}(q-1)}{2}$ unipotent orbits of order $q^{2(n-\delta)-1}(q-1)$ and exactly $2q^{2(n-\delta)-1}$ unipotent orbits of order $\frac{q^{2(n-\delta-1)}(q^2-1)}{2}$.
So, all together, with the additional $q^n$ orbits of size $1$ corresponding to all scalar matrices, we have 
\begin{multline*}
  q^n+\sum_{\delta=0}^{n-1}q^\delta \left(\frac{q^{2(n-\delta)-1}(q-1)}{2} q^{2(n-\delta)-1}(q+1) + \right. \\ \left.\frac{q^{2(n-\delta)-1}(q-1)}{2} q^{2(n-\delta)-1}(q-1) + 2q^{2(n-\delta)-1} \frac{q^{2(n-\delta-1)}(q^2-1)}{2}\right)=q^{4n}
  \end{multline*}
  elements. This proves that these are exactly all the orbits in $M_2(R)$.
\end{proof}

\bigskip

{\bf Statements and Declarations} \\

The author states that there are no competing interests. 

\bigskip

\bibliographystyle{amsplain}
\bibliography{biblio}

\bigskip

\end{document}